\documentclass[12pt]{article}
\usepackage{amsmath}
\usepackage{amsfonts}
\usepackage{amssymb}
\usepackage{mathtools}
\usepackage{accents}
\usepackage{tikz-cd}
\usepackage{graphicx}%
\providecommand{\U}[1]{\protect\rule{.1in}{.1in}}
\newtheorem{theorem}{Theorem}[section]

\newtheorem{corollary}[theorem]{Corollary}

\newtheorem{example}[theorem]{Example}

\newtheorem{proposition}[theorem]{Proposition}
\newtheorem{remark}[theorem]{Remark}

\newenvironment{proof}[1][Proof]{\noindent\textbf{#1.} }{\ \rule{0.5em}{0.5em}}
\begin{document}

\title{Topological Frobenius reciprocity and invariant hermitian forms}
\author{Tim Bratten \\Facultad de Ciencias Exactas, UNICEN \\7000, Tandil, Argentina.\\clarbratten@gmail.com
\and Mauro Natale \\Facultad de Ciencias Exactas, UNICEN\\7000, Tandil, Argentina. \\  natale.doc@gmail.com}
\date{}
\maketitle

\begin{abstract}

In his article \emph{Unitary Representations and Complex Analysis}, David Vogan gives a characterization of the continuous invariant Hermitian forms defined on the compactly supported sheaf cohomology groups of certain homogeneous analytic sheaves defined on open orbits in generalized flag manifolds. In the last section of the manuscript, Vogan raises a question about the possibility of a topological Frobenius reciprocity for these representations. In this article we give a specific version of the topological reciprocity in the regular, antidominant case and use it to study the existence of continuous invariant hermitian forms on the sheaf cohomology. In particular, we obtain a natural relationship between invariant forms on the sheaf cohomology and invariant forms on the geometric fiber. 

\end{abstract}

\emph{Keywords:} Representations of reductive Lie groups, Frobenius reciprocity, localization of representations. 

\emph{MSC2010}: 22E45, 22E46, 14L30, 14M15.

\section{Introduction}

In the article \cite{vogan}, D. Vogan gives a characterization of the continuous invariant Hermitian forms that exist on some geometrically defined representations that were studied in \cite{hecht1} and \cite{bratten1}. In the last section of the manuscript, Vogan raises a question (Question 10.2) about the possibility of a topological Frobenius reciprocity for these representations. Motivated by Vogan's question, in this manusrcipt we establish a specific version of the reciprocity, in the regular antidominant case, and use it to study the continuous invariant Hermitian forms on the nonzero compactly supported sheaf cohomology group of a homogeneous analytic sheaf. A natural application of the reciprocity is an analysis of the relationship between the invariant forms defined on sheaf cohomology and invariant forms defined on the geometric fiber. In order to unpack this relationship, we need to understand a corresponding Lie algebra cohomology group defined on the conjugate point. A key ingredient in our analysis is a  generalization of a duality theorem that appears in \cite{zabcic} and \cite{chang}. 

To treat our results in more detail, we introduce the basic set up. The proofs we give for reciprocity work for a reductive Lie group $G_{0}$ of Harish-Chandra class. However, for the sake of the introduction, we will assume that $G_{0}$ is a real form of a connected complex reductive group $G$. By this we mean: $G_0$ is a closed subgroup of $G$, $G_0$ has finitely many connected components, and the Lie algebra $\mathfrak{g}_0$ of $G_0$ is a real form of the Lie algebra, $\mathfrak{g}$ of $G$. In general, $P \subseteq G$ will denote a parabolic subgroup of $G$ and  \[Y = G/P \] will be the corresponding generalized flag manifold. Since a parabolic subgroup is equal to its normalizer in $G$, points in $Y$ are naturally identified with the parabolic subgroups conjugate to $P$. In the process of our proof, we will consider $Y$ as both an algebraic variety and a complex analytic manifold. However, for now, $Y$ is a holomorphic homogeneous space for the complex analytic group $G$. Adopting Vogan's nomenclature, a parabolic subgroup $P$ will be called \emph{nice} if it contains a Levi factor $L$ such that \[ G_{0} \cap P = G_{0} \cap L .\] In this case \[L_{0} = G_{0} \cap L\] is a real form of $L$ called the \emph{real Levi subgroup}. A $G_{0}$-orbit in $Y$ is  \emph{nice} when it's the orbit of a nice parabolic subgroup. Nice orbits are open in $Y$. We remark that every open orbit on the full flag manifold \[X  =  G/B \] of Borel subgroups, $B \subseteq G$, is nice and that sometimes, open orbits on generalized flag manifolds are not nice. 

We recall a few facts about the representations of reductive Lie group. Let $K_{0} \subseteq G_{0}$ be a maximal compact subgroup and let $K \subseteq G$ be its complexification. A \emph{Harish-Chandra module} $M$ can be defined as a finite-length $\mathfrak{g}$-module with compatible, algebraic $K$-action. We remark that the structure of $M$ as a $(\mathfrak{g}_{0}, K_{0})$-module uniquely determines the $\mathfrak{g}$ and $K$-actions, since both actions are obtained by complexification. A finite-length, admissible representation of $G_{0}$ on a complete, locally convex space is called  \emph{a globalization of} $M$ if the  subspace of $K_{0}$-finite vectors in the representation is isomorphic to $M$. We use the notation $M_{ \text{glob}}$ to indicate the globalization of a Harish-Chandra module $M$. In this manuscript we will focus on two canonical globalizations: Schmid's minimal and maximal globalizations \cite{schmid}. Both globalizations define exact functors and are related by the natural isomorphism \[M_{\text{max}} \cong ((M^{\vee})_{\text{min}})^{\prime} \] where $M^{\vee}$ is the $K$-finite dual of the Harish-Chandra module $M$ and the prime indicates the continuous dual with the strong topology. In general, there are continuous $G_{0}$-equivariant inclusions  \[ M_{\text{min}} \rightarrow M_{\text{glob}} \rightarrow M_{ \text{max}} .\]

Suppose $P$ is a nice parabolic subgroup and let $S = G_0 \cdot P$ be the $G_0$-orbit of $P$ in $Y = G/P$. There are certain points in $S$ that work well with respect to some algebraic constructions involving the choice of  maximal compact subgroup $K_0$ and its complexification $K$. Vogan calls these \emph{the very nice parabolic subgroups} and they can be characterized as the points in $S$ such that $K_0 \cap P = K_0 \cap L_0$ is a maximal compact subgroup of the real Levi factor. Both $K_0$ and its complexification $K$ act transitively on the very nice parabolic subgroups and they form the unique $K$-orbit $Q$ in $Y$ that satisfies $Q \subseteq S$ \cite{wolf}. The \emph{vanishing number}, $q$, \emph{of} $S$, is defined as the complex codimension of $Q$. 

Now let $V_{\text{min}}$ be a minimal globalization for $L_{0}$. Then, generalizing the finite-dimensional case, when $V = V_{\text{min}}$,  one can define a corresponding $G_{0}$-equivariant analytic sheaf $\mathcal{O} ( P, V_{\text{min}} )$ on $S$ \cite[Section~6]{bratten1}. In the infinite-dimensional case it is not clear this sheaf is locally free over the structure sheaf of holomorphic functions, however, the geometric fiber is isomorphic to $V_{min}$. It follows from the work in \cite{bratten1}, that the compactly supported sheaf cohomology groups \[H_{\text{c}}^{p} (S ,\mathcal{O} ( P, V_{\text{min}} )) \] are minimal globalizations for $G_{0}$ and that they vanish if $p < q$. On the other hand, let $\mathfrak{u}$ be the Lie algebra of the unipotent radical $U$ of $P$. When $W_{\text{max}}$ is a maximal globalization for $G_{0}$ then it turns out the Lie algebra homology groups \[H_{n} (\mathfrak{u} , W_{\text{max}})\] are maximal globalizations for $L_{0}$ \cite{bratten3}. In fact, when $P$ is a very nice parabolic subgroup, then there is a natural isomorphism  \[ H_{n} (\mathfrak{u} , W)_{\text{max}} \cong H_{n} (\mathfrak{u} , W_{\text{max}}) .\] Vogan's Question 10.2 asks about the relationship between  \[  \text{Hom}_{G_{0}} \left( H_{\text{c}}^{p} (S ,\mathcal{O} ( P, V_{\text{min}} )) , W_{\text{max}} \right)  \text{    and    }  \text{Hom}_{L_{0}} \left(  V_{\text{min}}  , H^{n} (\mathfrak{u} , W_{\text{max}}) \right)   \]   where $H^{n} (\mathfrak{u} , W_{\text{max}})$ denotes the Lie algebra cohomology group. 

In order to formulate our treatment of this relationship, we need to introduce the calculus of infinitesimal characters. Let $Z ( \mathfrak{g} )$ denote the center of the enveloping algebra $U ( \mathfrak{ g } )$. A $\mathfrak{g}$-\emph{infinitesimal character} is a morphism of algebras $ \Theta : Z ( \mathfrak{g} ) \rightarrow \mathbb{C} . $ A $\mathfrak{g}$-module is called \emph{quasisimple} if $Z ( \mathfrak{g} )$ acts by an infinitesimal character. In general, we call $M_{\text{glob}}$ \emph{quasisimple} if the corresponding Harish-Chabdra module $M$ is a quasisimple $\mathfrak{g}$-module. In this case $M_{\text{min}}$ and  $M_{\text{max}}$ are both quasisimple $\mathfrak{g}$-modules. One key property of our geometric construction of representations is that it preserves the property of quasisimple in a specific way. In particular, if $V_{\text{min}}$ is a minimal globalization for $L_{0}$ with $Z ( \mathfrak{l})$-infinitesimal character $\sigma : Z ( \mathfrak{l} ) \longrightarrow \mathbb{C} $,   then the sheaf cohomology groups $H_{\text{c}}^{p} (S ,\mathcal{O} ( V_{\text{min}} ))$ will have the infinitesimal character $ \Theta$ given by composition of $\sigma$ with the Harish-Chandra map \[ Z ( \mathfrak{g} ) \rightarrow Z ( \mathfrak{l} ) .\] In turn, if $\mathfrak{h} \subseteq \mathfrak{l} $ is a Cartan subalgebra, we have (unnormalized) Harish-Chandra maps $ Z ( \mathfrak{l} ) \rightarrow U (\mathfrak{h} )$ and  $Z ( \mathfrak{g} ) \rightarrow U (\mathfrak{h} )$, and a commutative diagram 

\[
\begin{tikzcd}
Z ( \mathfrak{g} ) \arrow{r} \arrow[swap]{dr} & Z ( \mathfrak{l} ) \arrow{d} \\
& U (\mathfrak{h} . ) 
\end{tikzcd} \]

Let $W$ be the Weyl group of $\mathfrak{h}$ in $\mathfrak{g}$ and let $W_{\mathfrak{l}}$ be the Weyl group of $\mathfrak{h} $ in $\mathfrak{l}$. Then the respective Harish-Chandra maps identify $\Theta$ with a $W$-orbit and $\sigma$ with a $W_{\mathfrak{l}}$-orbit in the dual space $\mathfrak{h}^{\ast}$. The $\mathfrak{l}$-infinitesimal character $\sigma$ is said to be \emph{regular} (with respect to $\mathfrak{g}$) if $W$ acts freely on the corresponding $W$-orbit in $\mathfrak{h}^{\ast}$. $\sigma$ is called \emph{antidominant} (with respect to parabolic subgroup $P$) if for any $\lambda$ in the corresponding $W_{\mathfrak{l}}$-orbit we have \[ \alpha^{\vee} (\lambda) \notin \lbrace 1, 2, 3, \cdots \rbrace\] for every root $\alpha$ of $\frak{h}$ in $\frak{u}$. In terms of our construction of representations we have the following: 

(1) When the infinitesimal character, $\sigma$ of $V_{\text{min}}$ is antidominant then 

 \[ H^p_{\text{c}} ( S , \mathcal{O} (P ,V_{\text{min}})) = 0 \]   unless $p = q$. 
 
(2) When the infinitesimal character of $V_{\text{min}}$ is antidominant and regular then $ H^q_{\text{c}} ( S , \mathcal{O} (P ,V_{\text{min}})) $ is irreducible when $ V_{\text{min}}$ is. 

\bigskip

\noindent In particular, if $\sigma$ is a regular, antidominant infinitesimal character for $L_0$ then \[ V_{\text{min}} \mapsto H_{\text{c}}^{q} (S ,\mathcal{O} ( P, V_{\text{min}} )) \] defines an exact (and, in fact, fully faithful) functor from the category of minimal globalizations for $L_0$ with infinitesimal character $\sigma$ to the category of quasisimple finite length admissible representations for $G_0$. 

\emph{The topological Frobenius reciprocity} is the fact that the geometric induction has a right adjoint given by the following construction. Let $\mathfrak{u}$ be the nilradical of the Lie algebra of $P$ and suppose $M_{\text{max}}$ is a quasisimple maximal globalization for $G_0$. Let $H_{q} ( \mathfrak{u} , M_{\text{max}} )_{\sigma}$ denote the corresponding $Z ( \mathfrak{l} )$ eigenspace in the Lie algebra homology group. Then the right adjoint to the functor of geometric induction is given by \[  M_{\text{glob}}  \mapsto  H_{q} ( \mathfrak{u} , M_{\text{max}} )_{\sigma} \] where $M_{\text{glob}}$ is a quasisimple finite length admissible representation for $G_0$. That is:  \[  \text{Hom}_{G_{0}} \left( H_{\text{c}}^{q} (S ,\mathcal{O} ( P, V_{\text{min}} )) , M_{\text{glob}} \right)  \cong  \text{Hom}_{L_{0}} \left(  V_{\text{min}}  , H_{q} (\mathfrak{u} , M_{\text{max}}) \right)   \]for every quasisimple finite length admissible representation $ M_{\text{glob}} $. For any nice parabolic subgroup $P$, we obtain the natural isomorphism 

\[  \text{Hom}_{G_{0}} \left( H_{\text{c}}^{q} (S ,\mathcal{O} ( P, V_{\text{min}} )) , M_{\text{max}} \right)  \cong  \text{Hom}_{L_{0}} \left(  V_{\text{min}}  , H_{q} (\mathfrak{u} , M_{\text{max}}) \right) . \] Dualizing, letting $s$ be the complex dimension of $Q$ and giving a completely formal definition to the representation $H^{s}  (S ,\mathcal{O} ( \mathfrak{p}, W_{\text{max}}) ) $  we obtain the natural isomorphism \[ \text{Hom}_{G_{0}} \left( M _{\text{min}} ,  H^{s}  (S ,\mathcal{O} ( \mathfrak{p}, W_{\text{max}}) \right) ) \cong  \text{Hom}_{L_{0}} \left( H_{s} (\mathfrak{u} , M _{\text{min}}) ,   W_{\text{max}}   \right) \] which is valid at least when $M _{\text{min}}$ is quasisimple and $W_{\text{max}} $ has an infinitesimal character that satisfies an appropriate regular dominant condition.

Our proof of the formula is basesd on the geometric construction of the Harish-Chandra module of $H_{\text{c}}^{q} (S ,\mathcal{O} ( P, V_{\text{min}} ))$ and depends on known properties of that construction,  although we've never previously encountered them put together this way.  The rest follows from results in \cite{bratten2} and \cite{bratten3}. We remark that this does give a very different proof than the one used by Vogan to derive analagous formulas for the cohomological parabolic induction in \cite{vogan2}. 
 
\bigskip 

Ostensibly, the topological Frobenius reciprocity qives a relationship between the continuous, invariant Hermitian forms defined on the representation $H_{\text{c}}^{q} (S ,\mathcal{O} ( P, V_{\text{min}} ))$ and the  continuous, invariant Hermitian forms defined on the geomteric fiber. To explain this, let $ (M_{\text{min}})^{h}$ denote the Hermitian dual of a minimal globalization $M_{\text{min}} $, as defined by Vogan in \cite{vogan}. Then the space of continuous invariant sesquilinear forms on $M_{\text{min}} $ is naturally identified with the finite-dimensional complex vector space \[  \text{Hom}_{G_{0}} \left( M_{\text{min}} , (M_{\text{min}})^{h} \right).   \] The subspace of continuous invariant Hermitian forms is a real form. Since \[  H_{p}(\mathfrak{u} , (M_{\text{min}})^{h}) \cong H^{p}(\mathfrak{u}^{\text{op}} , M_{\text{min}}) ^{h}  \] for each $p$, where $ H^{p}(\mathfrak{u}^{\text{op}} , M_{\text{min}})$ is the corresponding Lie algebra cohomology group defined for the nilradical of the conjugate (or opposite) parabolic subalgebra, letting $I(P,V)_{\text{min}} $ denote the representation $ H_{\text{c}}^{q} (S ,\mathcal{O} ( P, V_{\text{min}} ))$, we obtain \[   \text{Hom}_{G_{0}} \left( I(P,V )_{\text{min}}  , (I(P,V )_{\text{min}})^{h} \right)  \cong  \text{Hom}_{L_{0}} \left(  V_{\text{min}}  , H_{q} (\mathfrak{u} ,  (I(P,V)_{\text{min}} )^{h} ) \right) \cong     \]

\[ \text{Hom}_{L_{0}} \left(  V_{\text{min}}  , H^{q} (\mathfrak{u}^{\text{op}} , I(P,V)_{\text{min}}  )^{h} \right) . \] Thus, we would like to understand the $\mathfrak{u}^{\text{op}} $-cohomology group \[ H^{q} (\mathfrak{u}^{\text{op}} ,  H_{\text{c}}^{q} (S ,\mathcal{O} ( P, V_{\text{min}} ))) .\] Our strategy for this problem involves extending a duality theorem for standard modules that appears in \cite{zabcic} and \cite{chang}. This allows us to conclude that \[ H^{q} (\mathfrak{u}^{\text{op}} ,  H_{\text{c}}^{q} (S ,\mathcal{O} ( P, V_{\text{min}} )))_{\sigma} \cong V_{\text{min}} \] at least when $V_{\text{min}}$ is irreducible (our proof for duality, which is adapted from \cite{chang}, depends on irreducibility). In this way, we provide a natural (and geometric) proof of the known result that a nondegenerate continuous invariant Hermitian form on $V_{\text{min}}$ induces a nondegenerate continuous invariant Hermitian form on $ H_{\text{c}}^{q} (S ,\mathcal{O} ( P, V_{\text{min}} ))$. Conversely, whether or not $V_{\text{min}}$ is irreducible, any nonzero continuous invariant Hermitian form on  $ H_{\text{c}}^{q} (S ,\mathcal{O} ( P, V_{\text{min}} ))$ induces a nonzero morphism between  $V_{\text{min}}$ and its Hermitian dual in $H^{q} (\mathfrak{u}^{\text{op}} ,  H_{\text{c}}^{q} (S ,\mathcal{O} ( P, V_{\text{min}} )))^h$. 

\bigskip

Our article is organized as follows. The first section is the introduction. In the second section we give a geometric proof of the algebraic reciprocity. In the third section we introduce the topological reciprocity and consider the effects of the Hermitian dual functor, introduced by Vogan in \cite{vogan}. We also give a fairly detailed example of how things work when $G_0$ is a connected complex reductive group and $S$ is the open orbit on a full flag space. In the final section we generalize Chang's duality theorem and use this to analyze the relationship between the continuous invariant Hermitian forms on the representation $H_{\text{c}}^{q} (S ,\mathcal{O} ( P, V_{\text{min}} ))$ and the continuous invariant Hermitian forms on the geometric fiber. 

\bigskip

\section{Algebraic Frobenius reciprocity}

In this section we give an argument for algebraic version of our reciprocity formula, using the localization theory to a generalized flag manifold. The algebraic reciprocity theorem is an adjointness property that holds, in general, for closed $K$-orbits in $Y$ (i.e. does not depend on the assumption that $P$ is nice.) The result follows from previously established facts about: (1) localization (the equivalence of categories given by localization and global sections), (2) the direct image for closed embeddings (Kashiwara's equivalence of categories) and (3) the identification of geometric fibers with $\mathfrak{u}$-homology. However, we are unaware of any version of this reciprocity in the published literature, so we provide a proof.  

A natural context for establishling the results in this section is to assume $K \subseteq G$ is an algebraic subgroup that has finitely many orbits on the full flag manifold (for example, this is the context used in \cite{milicic}). We need to introduce the machinery of the localization theory and review the corresponding construction of the standard Harish-Chandra modules. 

Let $X$ be the full flag manifold of Borel subgroups of $G$, $P \subseteq G$ a parabolic subgroup and $Y = G /P$ the corresponding generalized flag manifold. Since each Borel subgroup of $P$ is conjugate in $P$, there is a canonical $G$-equivariant projection $\pi : X \rightarrow Y $. 

In this section we treat all objects in the algebraic category. In particular, $X$ and $Y$ are smooth, complex, algebraic varieties with their accompanying structure sheaves $\mathcal{O}_{X}$ and $\mathcal{O}_{Y}$.  For each  $x \in X$, let $\mathfrak{b}_{x}$ be the corresponding Borel subalgebra of $\mathfrak{g}$ and let   
\[ \mathfrak{n}_{x} = [ \mathfrak{b}_{x},\mathfrak{b}_{x} ] \] be the nilradical of $ \mathfrak{b}_{x} $. The adjoint action of the Borel subgroup $B_{x}$ on $\mathfrak{b}_{x}$ and $\mathfrak{n}_{x} $ determines homogeneous algebraic vector bundles on X. We let $\mathfrak{b}^{\bullet}$  and $\mathfrak{n}^{\bullet}$ denote the corresponding sheaves of sections. Since $B_{x}$ acts trivially on the quotient $ \mathfrak{b}_{x} / \mathfrak{n}_{x}$, the sheaf $ \mathfrak{b}^{\bullet} / \mathfrak{n}^{\bullet}$  is a free $\mathcal{O}_{X}$-module and the global sections \[ \mathfrak{h} = \Gamma ( X ,\mathfrak{b}^{\bullet} / \mathfrak{n}^{\bullet} ) \] are a finite-dimensional, abelian Lie algebra, isomorphic to a Cartan subalgebra $\mathfrak{c}$ of $\mathfrak{g}$ contained in $\mathfrak{b}_{x}$, via the \emph{specialization to x:}         \[ \mathfrak{c} \rightarrow \mathfrak{b}_{x} / \mathfrak{n}_{x}  \leftarrow  \mathfrak{h} . \] We call $\mathfrak{h}$ the \emph{abstract Cartan subalgebra}. Let $\mathfrak{h}^{\ast}$ be the complex dual of $\mathfrak{h}$. There is a set of \emph{abstract roots} $\Sigma \subseteq \mathfrak{h}^{\ast}$ and a corresponding Weyl group $W$. \emph{The set of positive roots} $\Sigma^{+} \subseteq \Sigma$ correspond to the roots of a Cartan subalgebra $\mathfrak{c}$ in $\mathfrak{b}_{x}$ via the specialization to $x$. 

Via the Harish-Chandra map, the Weyl group orbits in $\mathfrak{h}^{\ast}$ parameterize the infinitesimal characters of $Z(\mathfrak{g})$. We will write $\lambda \in \Theta$ if the Weyl group orbit $W \cdot \lambda$ corresponds to the infinitesimal character $\Theta$. Let $U_{\Theta}$ denote the quotient of $U(\mathfrak{g})$ by the ideal generated from the kernel of $\Theta$. For each $\lambda \in \mathfrak{h}^{\ast}$, Beilinson and Bernstein define a twisted sheaf of differential operators $\mathcal{D}_{\lambda}$ on $X$ \cite{bernstein}. They show that $\mathcal{D}_{\lambda}$ is acyclic for global sections and that \[ \Gamma (X , {D}_{\lambda}) \cong U_{\Theta} \] when $\lambda \in \Theta$. On the other hand, if $M$ is a $U_{\Theta}$-module, then for each $\lambda \in \mathfrak{h}^{\ast}$, there is a localization functor, $\Delta_{\lambda}$, is defined by \[ \Delta_{\lambda}(M) = \mathcal{D}_{\lambda} \otimes_{U_{\Theta}} M .\] Localization is left adjoint to the functor of global sections. 

An infinitesimal character is called \emph{regular} if the Weyl group acts freely on the corresponding orbit in $\mathfrak{h}^{\ast}$. An element $\lambda \in \mathfrak{h}^{\ast}$ is called \emph{antidominant} if $\alpha^{\vee} (\lambda)$ is not a positive integer, for every positive coroot $\alpha^{\vee}$. When $\lambda$ is antidominant and regular then localization and global sections define an equivalence of categories between the category of Harish-Chandra modules with infinitesimal character $\Theta$ and the category of $K$-equivariant coherent sheaves of $\mathcal{D}_{\lambda}$-modules (called the \emph{Harish-Chandra sheaves on $X$}). 

We now consider how these results carry over to a generalized flag manifold, $Y ,$ in a way that is suitable for our purposes. The set up is the same as in \cite{bratten1}. For each $y \in Y$, $\mathfrak{p}_{y}$ is the corresponding parabolic subalgebra and $\mathfrak{u}_{y}$ will be the nilradical of $\mathfrak{p}_{y}$. The action of the parabolic subgroup $P_{y}$ on the quotient \[ \mathfrak{l}_{y} = \mathfrak{p}_{y} / \mathfrak{u}_{y} \]  is not trivial in this more general setting. We introduce the $G$-equivariant sheaf of algebraic sections, $U(\mathfrak{l}^{\bullet})$, corresponding to the $P_{y}$ action on the universal enveloping algbera $U(\mathfrak{l}_{y})$. This is a locally free sheaf of $\mathcal{O}_{Y}$-modules and a sheaf of algebras with the pointwise multiplication. Hence, the center $Z(\mathfrak{l}^{\bullet})$ of $U(\mathfrak{l}^{\bullet})$ is a free sheaf of $\mathcal{O}_{Y}$-modules and the global sections \[ \mathcal{Z}(\mathfrak{l}) = \Gamma (Y , Z(\mathfrak{l}^{\bullet})) \] will be isomorphic to the center,  $Z(\mathfrak{l})$, of the enveloping algebra of a Levi factor $ \mathfrak{l} \subseteq \mathfrak{p}_{y}$ via the specialization to a point y: \[ \mathcal{Z}(\mathfrak{l} ) \rightarrow Z(\mathfrak{l}_{y}) \leftarrow Z(\mathfrak{l}). \] Since the natural projection identifies the Borel subalgebras of $\mathfrak{g}$ contained in $\mathfrak{p}_{y}$ with the Borel subalgebras of $\mathfrak{l}_{y}$, we can use specialization to identify a set of abstract roots of the Levi factor $\Sigma(\mathfrak{l} ) \subseteq  \Sigma $ and a corresponding set of positive roots $\Sigma(\mathfrak{l} )^{+} \subseteq  \Sigma^{+} .$ We let $W_{\mathfrak{l}} \subseteq W$ be the Weyl group of the root system $\Sigma(\mathfrak{l} )$. Since $X_{y} = \pi^{-1} ( \lbrace y \rbrace ) $ is the flag manifold for $\mathfrak {l}_{y}$, the restriction identifies $\mathfrak{h}$ with the abstract Cartan of $\mathfrak{l}_{y}$. There is an unnormalized Harish-Chandra map $ \mathcal{Z}(\mathfrak{l}) \rightarrow U( \mathfrak{h} ) $. We parametrize the infinitesimal characters of $ \mathcal{Z}(\mathfrak{l}) $ via the composition \[ \mathcal{Z}(\mathfrak{l}) \rightarrow U( \mathfrak{h} ) {\xrightarrow{\lambda + \rho}}  \mathbb{C} \] where $\lambda \in \mathfrak{h}^{\ast}$ and $\rho$ denotes one-half the sum of the roots in $\Sigma^{+}$. In this way, the $ \mathcal{Z}(\mathfrak{l})$-infinitesimal characters are parameterized by $W_{\mathfrak{l}}$-orbits in $\mathfrak{h}^{\ast} $. Occasionally we use the notation $\sigma_{\lambda}$ to indicate the corresponding infinitesimal character, even though $\sigma_{\lambda} = \lambda + \rho$ when $\mathfrak{l}_{y} = \mathfrak{b}_{x} /  \mathfrak{n}_{x} .$ Since the unnormalized Harish-Chandra map sends $Z (\mathfrak{g} )$  into $ \mathcal{Z}(\mathfrak{l})$, a $\mathfrak{g}$-infinitesimal character $\Theta$. with $\lambda \in \Theta$, corresponds to the family of $ \mathcal{Z}(\mathfrak{l})$-infinitesimal characters  \[ \lbrace \sigma_{[w]\lambda} : [w] \in W / W_{\mathfrak{l}} \rbrace .\]  Each member of this family will be called \emph{regular} if the corresponding $Z (\mathfrak{g} )$-infinitesimal character is regular. An element $\lambda \in \mathfrak{h}^{\ast}$ will be called \emph{antidominant for} $Y$ if there is an element in the orbit $W_{\mathfrak{l}} \cdot \lambda$ that is antidominant. Equivalently we could require that $ \alpha^{\vee}( \lambda) $ not be a positive integer for each $\alpha \in  \Sigma^{+}  - \Sigma^{+} (\mathfrak{l} )$. When this condition holds, we also refer to the $ \mathcal{Z}(\mathfrak{l})$-infinitesimal character $\sigma_{\lambda}$ as being antidominant.

We consider the sheaf of algebras $\pi_{\ast} \mathcal{D}_{\lambda}$ on $Y$. One knows that  ${D}_{\lambda}$ is acylic for $\pi_{\ast}$ \cite{chang}. By the Leray spectral sequence, it follows that $\pi_{\ast} \mathcal{D}_{\lambda}$ is a acylic for the global sections on $Y$.  There is a morphism of sheaves of algebras  $Z(\mathfrak{l}^{\bullet}) \rightarrow \pi_{\ast} \mathcal{D}_{\lambda}$. Via the restriction map, $ \mathcal{Z}(\mathfrak{l})$ acts on the geometric fiber of a sheaf of $\pi_{\ast} \mathcal{D}_{\lambda} $-modules, by the infinitesimal character $\sigma_{\lambda}$ and  \[ \pi_{\ast} \mathcal{D}_{\lambda} \cong \pi_{\ast} \mathcal{D}_{ w \lambda}   \text{    for each    } w \in W_{\mathfrak{l}} .\] Given a $U_{\Theta}$-module $M$ and a $W_{\mathfrak{\l}}$-orbit  $W_{\mathfrak{\l}} \cdot \lambda$ for $\lambda \in \Theta $ (i.e. a $ \mathcal{Z}(\mathfrak{l})$-infinitesimal character $\sigma_{\lambda}$), we can define a localization functor \[ \Delta_{Y}(M) = \pi_{\ast} \mathcal{D}_{\lambda} \otimes_{U_{\Theta}} M .\] Localization is left adjoint to the functor of global sections. When $\sigma_{\lambda}$ is regular and antidominant, then localization and global sections define an equivalence of categories between the Harish-Chandra modules with infinitesimal character $\Theta$ and the $K$-equivariant coherent sheaves of $ \pi_{\ast} \mathcal{D}_{\lambda}$-modules (the Harish-Chandra sheaves on $Y$) \cite{chang}. 

We review the construction of the standard Harish-Chandra sheaves from \cite{chang} and used in \cite{bratten1}. Let \[ i : Q \rightarrow Y\] be the inclusion and begin by assuming $Q$ is an arbitrary $K$-orbit. Choose $y \in Q$ and suppose $V$ is a Harish-Chandra module for $(\mathfrak{l}_{y} ,K_{y})$ where $K_{y}$ is the stabilizer of $y$  in $K$ and we also assume $ \mathcal{Z}(\mathfrak{l})$ acts  by an infinitesimal character $\sigma$. The Harish-Chandra module $V$ determines a $K$-equivariant algebraic vector bundle over $Q$. The corresponding sheaf of sections $\mathcal{V}$ is a locally free sheaf of $\mathcal{O}_{Q}$ -modules. Let $U_{Q}(\mathfrak{l}^{\bullet})$ be the $K$-equivariant locally free sheaf corresponding to the $K_{y}$-action on $U(\mathfrak{l}_{y})$ and let $\mathfrak{k}$ be the Lie algebra of $K$. The action of $\mathfrak{l}_{y}$ on $V$ determines a pointwise action of $U_{Q}(\mathfrak{l}^{\bullet})$ on $\mathcal{V}$ and, by differentiating the $K$-action, we obtain an action of $U(\mathfrak{k}^{\bullet})$ on  $\mathcal{V}$. Supppse $ \lambda \in \sigma$. In turn, the actions of $U_{Q}(\mathfrak{l}^{\bullet})$ and $U(\mathfrak{k}^{\bullet})$ determine an action for a sheaf of algebras, $(\pi_{\ast}\mathcal{D}_{\lambda})^{i}$, and there is a corresponding direct image functor \[ \mathcal{V} \mapsto i_{+} \mathcal{V} \] that sends Harish-Chandra sheaves for $( (\pi_{\ast} (\mathcal{D}_{\lambda} )^{i}, K)$ to Harish-Chandra sheaves for  $( \pi_{\ast} (\mathcal{D}_{\lambda}, K)$. When the orbit $Q$ is affinely embedded then the direct image is exact. In particular, if the infinitesimal character $\sigma$ is regular and antidominant for $Y$ then the functor \[ V \mapsto \Gamma (X,i_{+} \mathcal{V}) \] is an exact functor from the category of Harish-Chandra modules, $\mathbf{M}_{\sigma} ( \mathfrak{l}_{y} , K_{y} )$, with $\mathfrak{l}_{y}$-infinitesimal character $\sigma$, to the category of Harish-Chandra modules,  $\mathbf{M}_{\Theta} ( \mathfrak{g} , K )$, with $\mathfrak{g}$-infinitesimal character $\Theta$. 

We now introduce the $\mathfrak{u}_{y}$-homology groups $H_{p} ( \mathfrak{u}_{y} , M )$ for a  Harish-Chandra module $M$ for $ ( \mathfrak{g} , K )$. These are Harish-Chandra modules for $  ( \mathfrak{l}_{y} , K_{y} )$ (a geometric argument for this can be given as in \cite[Section 7]{bratten1}). When $M$ is in  $\mathbf{M}_{\Theta} ( \mathfrak{g} , K )$ with $\Theta$ regular, then the homology groups $H_{p} ( \mathfrak{u}_{y} , M )$ split as a direct sum of eigenspaces for the $ \mathcal{Z}(\mathfrak{l})$-infinitesimal characters coming from the compatible family mentioned previously \cite{bratten3}. Let \[ H_{p} ( \mathfrak{u}_{y} , M )_{\sigma} \] denote the $ \mathcal{Z}(\mathfrak{l})$-eigenspace for the infinitesimal character $\sigma$. The algebraic Frobenius reciprocity is the following. 

\begin{theorem}

Maintain the previous notations. Assume $Q$ is a closed orbit and let $q$ be the codimension of $Q$ in $Y$. Suppose $\sigma$ is a $ \mathcal{Z}(\mathfrak{l})$-infinitesimal character that is regular and antidominant for $Y$. Then the exact functor \[ V \mapsto \Gamma (X,i_{+} \mathcal{V}) \] from the category $\mathbf{M}_{\sigma} ( \mathfrak{l}_{y} , K_{y} )$ to the category $\mathbf{M}_{\Theta} ( \mathfrak{g} , K )$ has right adjoint \[ M \mapsto H_{q} ( \mathfrak{u}_{y} , M )_{\sigma} .\]

\end{theorem}

\begin{proof}

Suppose $\lambda \in \sigma$ and let \[ \Delta_{Y}(M) = \pi_{\ast} \mathcal{D}_{\lambda} \otimes_{U_{\Theta}} M \] be the localization functor. If $\mathcal{F}$ is a sheaf of $\mathcal{O}_{Y}$-modules then let \[T_{y}\mathcal{F} = \mathbb{C} \otimes_{(\mathcal{O}_{Y})_{y}} \mathcal{F}_{y} \] denote the geometric fiber of $\mathcal{F}$ at the point $y$. Since $\Theta$ is regular it follows \cite{bratten3} that for any $M$ in $\mathbf{M}_{\Theta} ( \mathfrak{g} , K )$ there is a natural isomorphism of Harish-Chandra modules \[ H_{q} ( \mathfrak{u}_{y} , M )_{\sigma} \cong L_{q} T_{y} \Delta_{Y}(M)  \] where $ L_{q} T_{y}$ denotes the $q$-th derived functor of the geometric fiber. Since $\sigma$ is regular and antidominant for $Y$ it   follows that global sections and localization are equivalences of categories. Therefore \[ \text{Hom}_{( \mathfrak{g} , K )} ( \Gamma (X,i_{+} \mathcal{V}) , M) \cong \text{Hom}_{( \pi_{\ast} \mathcal{D}_{\lambda} , K )}  ( i_{+} \mathcal{V}  , \Delta_{Y}(M)). \] Thus it suffices to show  \[ \text{Hom}_{( \pi_{\ast} \mathcal{D}_{\lambda} , K )}  ( i_{+} \mathcal{V}  , \Delta_{Y}(M))  \cong \text{Hom}_{( \mathfrak{l}_{y} , K_{y} ) } ( V , L_{q} T_{y} \Delta_{Y}(M) ) . \] It turns out this last identity can be deduced from Kashiwara's equivalence of categories as follows. When $\mathcal{F}$ is a sheaf of Harish-Chandra modules then the inverse image in the category category of $\mathcal{O}$-modules \[ i^{\ast}\mathcal{F} = \mathcal{O}_{Q}  \otimes_{i^{-1} \mathcal{O}_{Y} }  i^{-1} \mathcal{F}\] is a Harish-Chandra sheaf for    
$( (\pi_{\ast} (\mathcal{D}_{\lambda} )^{i}, K)$, as are the derived inverse images $ L_{p}i^{\ast } \mathcal{F}$ which vanish for $p > q$. One knows \[ L_{q}i^{\ast } i_{+} \mathcal{V}  \cong  \mathcal{V}    \text{    and    }  i_{+} L_{q}i^{\ast}  \mathcal{F}  \cong  \Gamma_{Q} \mathcal{F}  \] where $ \Gamma_{Q} \mathcal{F}$ denotes the sheaf of sections with support in $Q$. It follows from these formulas that 
the functor $L_{q}i^{\ast }$  is right adjoint to the direct image, so that 

\[ \text{Hom}_{( \pi_{\ast} \mathcal{D}_{\lambda} , K )}  ( i_{+} \mathcal{V}  , \Delta_{Y}(M))  \cong \text{Hom}_{( (\pi_{\ast} \mathcal{D}_{\lambda})^{i} , K )}  ( \mathcal{V}  , L_{q}i^{\ast } \Delta_{Y}(M)) . \] Since $ L_{q}i^{\ast } \Delta_{Y}(M)) $ is a Harish-Chandra sheaf on a $K$-orbit, it is a locally free sheaf of $\mathcal{O}_{Q}$-modules so that morphisms are completely determined by morphisms of the geometric fiber. Thus \[ \text{Hom}_{( (\pi_{\ast} \mathcal{D}_{\lambda})^{i} , K )}  ( \mathcal{V}  , L_{q}i^{\ast } \Delta_{Y}(M)) \cong \text{Hom}_{ (\mathfrak{l}_{y} , K_{y} )} ( V, T^{Q}_{y}L_{q}i^{\ast } \Delta_{Y}(M))  \] where $T^{Q}_{y}$ is the geometric fiber for a sheaf of $\mathcal{O}_{Q}$-modules. Thus it suffices to prove \[T^{Q}_{y}L_{q}i^{\ast } \Delta_{Y}(M) \cong  L_{q} T_{y} \Delta_{Y}(M) .\] But this last point follows from a standard homological algebra argument, since $ T^{Q}_{y} \circ i^{\ast} \cong T_{y}$ and since each of the sheaves $L_{p} i^{\ast }  \Delta_{Y}(M)$, for $0 \leq p \leq q$, is a Harish-Chandra sheaf of $( (\pi_{\ast} (\mathcal{D}_{\lambda} )^{i}, K)$-modules and thus a locally free sheaf of $ \mathcal{O}_{Q}$-modules that is acyclic for the functor $ T^{Q}_{y} $.

\end{proof}

\begin{remark}

Maintaining the assumptions of the previous theorem, we obtain the formula \[ \text{Hom}_{( \mathfrak{g} , K )} ( \Gamma (X,i_{+} \mathcal{V}) , M) \cong  \text{Hom}_{( \mathfrak{l}_{y} , K_{y} ) } ( V , H_{q} ( \mathfrak{u}_{y} , M ) ) \] for any quasisimple Harish-Chandra module $M$. For a general Harish-Chandra module $M$ we can introduce the $Z (\mathfrak{g} )$-eigenspace $M_{\Theta}$ into the formula: \[ \text{Hom}_{( \mathfrak{g} , K )} ( \Gamma (X,i_{+} \mathcal{V}) , M) \cong  \text{Hom}_{( \mathfrak{g} , K )} ( \Gamma (X,i_{+} \mathcal{V}) , M_{\Theta}) \cong \text{Hom}_{( \mathfrak{l}_{y} , K_{y} ) } ( V , H_{q} ( \mathfrak{u}_{y} , M_{\Theta} ) ) .  \]

$\blacksquare$

\end{remark}

\begin{remark}

Now supppose that $K$ is the complexification of a maximal compact subgroup $K_{0}$ of $G_{0}$ and $P_{y}$ is a very nice parabolic subgroup with real Levi subgroup $L_{0}$. We introduce the $K_{0}$-finite dual \[ M \mapsto M^{\vee} \] on the category of Harish-Chandra modules (we use the same notation for the $K_{0} \cap L_{0}$-finite dual). Then one knows \cite{bratten3}:  

\[H_{q} ( \mathfrak{u}_{y} , M^{\vee} ) \cong  H^{q} ( \mathfrak{u}_{y} , M )^{\vee} \] where $ H^{q} ( \mathfrak{u}_{y} , M ) $ denotes the $q$-th Lie algebra cohomology group. Using the duality theorem in \cite{chang2} and the fact that the $K_{0}$-finite dual is exact, one can transpose the result in Theorem 2.1 into a reciprocity theorem for the cohomological parabolic induction, where the $(n-q)$-th Lie algebra homology group defines the left adjoint to the induction functor. We leave the details to the reader.  

$\blacksquare$

\end{remark}

\section{Topological reciprocity and the Hermitian dual}

In this section $G_0$ is a reductive group of Harish-Chandra class with maximal compact subgroup $K_0$. We begin with the topological reciprocity. 

\begin{theorem}

Suppose $\mathfrak{p}$ is a very nice parabolic subalgebra, $L_{0}$ is the real Levi subgroup, $\mathfrak{u}$ is the nilradical of $\mathfrak{p}$, $S$ is the $G_{0}$-orbit of $\mathfrak{p}$ in the corresponding flag manifold $Y$ and $q$ is the vanishing number of $S$. Let $V_{min}$ be a minimal globalization of $L_{0}$ that has  a regular, antidominant infinitesimal character. Then for every quasisimple finite length admissible representation $M_{\text{glob}}$ for $G_{0}$ there is a natural isomorphism \[  \text{Hom}_{G_{0}} \left( H_{\text{c}}^{q} (S ,\mathcal{O} ( \mathfrak{p}, V_{\text{min}} )) , M_{\text{glob}} \right)  \cong  \text{Hom}_{L_{0}} \left(  V_{\text{min}}  , H_{q} (\mathfrak{u} , M_{\text{max}}) \right) .   \]

\end{theorem}

\begin{proof}

Let $y \in Y$ be the point corresponding to the parabolic subalgebra $\mathfrak{p}$. Since $K_0 \cap L_0$ is a maximal compact subgroup of $L_0$ and since $K$ is the complexification of $K_0$, referring to the notation in the previous section, $K_y$ is a parabolic subgroup of $K$ with a Levi factor that is the complexification of $K_0 \cap L_0$. Thus $V_{\text{min}}$ is the minimal globalization of a Harish-Chandra module, $V$ for $(\mathfrak{l}_{y} , K_{y})$, that has a  regular antidominant infinitesimal character. Using the notation from the previous section, let $\Gamma (X, i_{+} \mathcal{V})$ denote the corresponding standard Harish-Chandra module and let $I (\mathfrak{p} ,V) $ denote the Harish-Chandra of $ H_{\text{c}}^{q} (S ,\mathcal{O} ( \mathfrak{p} , V_{\text{min}} ))$. Then one of the main points of \cite{bratten1} is that \[ \Gamma (X, i_{+}  \mathcal{V})  \cong I (\mathfrak{p},V). \] On the other hand, the representation $ M_{\text{glob}} $ is the globalization of a quasisimple Harish-Chandra module $M$ for $ (\mathfrak{g} , K )$. Therefore \[ \text{Hom}_{G_{0}} \left( H_{\text{c}}^{q} (S ,\mathcal{O} ( \mathfrak{p} , V_{\text{min}} )) , M_{\text{glob}}   \right)  \cong  \text{Hom}_{ ( \mathfrak{g} , K )} ( I(\mathfrak{p},V) , M)  \cong  \] 

\[ \text{Hom}_{( \mathfrak{g} , K )} ( \Gamma (X,i_{+} \mathcal{V}) , M) \cong  \text{Hom}_{( \mathfrak{l}_{y} , K_{y} ) } ( V , H_{q} ( \mathfrak{u} , M ) \cong \]

\[ \text{Hom}_{L_{0}} \left(  V_{\text{min}}  , H_{q} (\mathfrak{u} , M )_{\text{max}} \right) \cong \text{Hom}_{L_{0}} \left(  V_{\text{min}}  , H_{q} (\mathfrak{u} , M_{\text{max}}) \right)  \] where the fact the maximal globalization commutes with the $\mathfrak{u}$-homolgy group is shown in \cite{bratten3}.

\end{proof}

If we replace $M_{\text{glob}}$ with $M_{\text{max}}$ in the above reciprocity formula, then it no longer depends on the identification of a Harish-Chandra module. Therefore we obtain the following. 

\begin{corollary} 

Suppose $\mathfrak{p}$ is a nice parabolic subalgebra, $L_{0}$ is the real Levi subgroup, $\mathfrak{u}$ is the nilradical of $\mathfrak{p}$, $S$ is the $G_{0}$-orbit of $\mathfrak{p}$ in the corresponding flag manifold $Y$ and $q$ is the vanishing number of $S$. Let $V_{min}$ be a minimal globalization of $L_{0}$ that has  a regular, antidominant infinitesimal character. Then for every quasisimple maximal globalization $M_{\text{max}}$, there is a natural isomorphism \[  \text{Hom}_{G_{0}} \left( H_{\text{c}}^{q} (S ,\mathcal{O} ( \mathfrak{p}, V_{\text{min}} )) , M_{\text{max}} \right)  \cong  \text{Hom}_{L_{0}} \left(  V_{\text{min}}  , H_{q} (\mathfrak{u} , M_{\text{max}}) \right) .   \]   

\end{corollary} 

\begin{remark}

The dual form of the topological Frobenius reciprocity, also referred to in Vogan's Question 10.2, is a formal consequence of the previous theorem. Indeed, since \[ \text{Hom}_{G_{0}} \left( M_{\text{min}} , N_{\text{max}} \right) \cong  \text{Hom}_{G_{0}} \left( (N_{\text{max}})^{\prime} ,  (M_{\text{min}})^{\prime}  \right) \] for an pair of representations, we have 

\[  \text{Hom}_{G_{0}} \left( (N_{\text{max}})^{\prime} ,  H_{\text{c}}^{q} (S ,\mathcal{O} ( \mathfrak{p}, V_{\text{min}} ))^{\prime}  \right)  \cong  \text{Hom}_{L_{0}} \left( H_{q} (\mathfrak{u} , N_{\text{max}})^{\prime} ,   (V_{\text{min}})^{\prime}   \right) \cong   \] 

\[ \text{Hom}_{L_{0}} \left( H^{q} (\mathfrak{u} , (N_{\text{max}})^{\prime}) ,   (V_{\text{min}})^{\prime}   \right). \] Replacing $ (N_{\text{max}})^{\prime}  $ with $M _{\text{min}}$ in the formula, and using the fact that \[ H^{q} (\mathfrak{u} , (N_{\text{max}})^{\prime}) \cong H_{s} (\mathfrak{u} , (N_{\text{max}})^{\prime}) \otimes \chi_{\frak{u}}^{-1} \] where $s = \dim_{\mathbb{C}} (Y) - q $ and $ \chi_{\frak{u}} $ is the determinant character for the adjoint action of $L_{0}$ on $\frak{u}$, we obtain \[ \text{Hom}_{G_{0}} \left( M _{\text{min}} ,  H_{\text{c}}^{q} (S ,\mathcal{O} ( \mathfrak{p} , V_{\text{min}} ))^{\prime}  \right)  \cong  \text{Hom}_{L_{0}} \left( H_{s} (\mathfrak{u} , M _{\text{min}}) ,   (V_{\text{min}})^{\prime} \otimes \chi_{\frak{u}}   \right) .\] Introducing  the perspicuous notation \[   H_{\text{c}}^{q} (S ,\mathcal{O} ( \mathfrak{p} , V_{\text{min}} ))^{\prime} =  H^{s}  (S ,\mathcal{O} ( \mathfrak{p}, (V_{\text{min}} )^{\prime} ) \otimes \chi_{\frak{u}} )  \] and replacing $(V_{\text{min}} )^{\prime} \otimes \chi_{\frak{u}}$ with $W_{\text{max}}$, we obtain the reciprocity formula \[ \text{Hom}_{G_{0}} \left( M _{\text{min}} ,  H^{s}  (S ,\mathcal{O} ( \mathfrak{p}, W_{\text{max}}) \right) ) \cong  \text{Hom}_{L_{0}} \left( H_{s} (\mathfrak{u} , M _{\text{min}}) ,   W_{\text{max}}   \right) \] which is valid when $M _{\text{min}}$ is quasisimple and $W_{\text{max}} $ has an infinitesimal character that satisifies the corresponding regular, dominant condition. When $W_{\text{max}} $ is finite-dimensional then Serre duality \cite{serre} implies that the representation $ H^{s}  (S ,\mathcal{O} ( \mathfrak{p}, W_{\text{max}}) )$ is the sheaf cohomology of a corresponding $G_{0}$-equivariant holomorphic vector bundle defined on $S$. However it is not obvious to us that this representation is realized as the sheaf cohomology group of some corresponding $G_0$-equivariant sheaf on $S$,  when $W_{\text{max}} $ is infinite-dimensional. Nevertheless, in the infinite-dimensional case, the representation $ H^{s}  (S ,\mathcal{O} ( \mathfrak{p}, W_{\text{max}}))$ can be realized as the hypercohomology of the dual of the complex used to calculate the sheaf cohomology for $H_{\text{c}}^{q} (S ,\mathcal{O} ( \mathfrak{p}, V_{\text{min}} ))$.

$\blacksquare$

\end{remark}

We now introduce the Hermitian dual $ (M_{\text{min}} )^{h}$ as defined by Vogan in \cite{vogan}. In particular, $ (M_{\text{min}})^{h}$ is the space of continuous, conjugate-linear forms on $M_{\text{min}}$ with its natural topology as a maximal globalization, i.e. conjugation on $\mathbb{C}$ defines an equivariant, conjugate-linear isomorphism between $ (M_{\text{min}})^{h}$ and the continuous dual $ (M_{\text{min}})^{\prime} \cong (V^{\vee})_{_{\text{max}}} $. The contravariant  functor \[M \mapsto (M_{\text{min}} )^{h} \] is exact. One can naturally extend the definition of Hermitian dual to maximal globalizations. In this way we obtain the identities  

\[ ((M _{\text{min}})^{h})^{h} \cong M _{\text{min}}  \; \text{ and } \; ((M _{\text{max}})^{h})^{h} \cong M _{\text{max}}  . \] 

\begin{proposition} 

Suppose $ M _{\text{min}} $ is a minimal globalization. Then the space of continuous invariant sesquilinear forms on $ M _{\text{min}} $ is naturally isomorphic to \[  \text{Hom}_{G_{0}} \left( M_{\text{min}} , (M_{\text{min}})^{h} \right)    \] and the corresponding space of continuous invariant Hermitian forms is a real form of $\text{Hom}_{G_{0}} \left( M_{\text{min}} , (M_{\text{min}})^{h} \right)$.

\end{proposition}

\begin{proof}

This is basically shown in Vogan's manuscript, but we sketch some details for completeness. 

If $\varphi \in \text{Hom}_{G_{0}} \left( M_{\text{min}} , (M_{\text{min}})^{h} \right)$ then \[ \langle m_{1} , m_{2} \rangle = (\varphi (m_{1})) (m_{2}) \] is a separately continuous invariant sesquilinear form and vice-verse. The fact that separately continuous sesquilinear forms are continuous in this context follows from a theorem of topological vector spaces. On the other hand, for any continuous invariant sesqulinear form, $ \langle \cdot , \cdot \rangle$, we can define the conjugate transpose $ \langle \cdot , \cdot \rangle^{\dagger}$ by the equation \[ \langle v , w \rangle^{\dagger} = \overline{ \langle w , v \rangle}  .\] This is another continuous invariant sesquilinear form and the original form is Hermitian if and only if \[ \langle \cdot , \cdot \rangle^{\dagger} = \langle \cdot , \cdot \rangle .\] Now suppose $\mu$ is a continuous invariant sesquilinear form and define \[ \text{re}(\mu) = \frac{ \mu + \mu^{\dagger}}{2}  \text{   and   } \text{im} (\mu) = \frac{i (\mu^{\dagger} - \mu )}{2}. \] Then each of these forms is a continuous invariant hermitian form and $\mu = \text{re}(\mu) + i \text{im} (\mu). $ 

\end{proof}

We now consider the effect of the Hermitian dual on the Lie algebra homology groups $H_{p} ( \mathfrak{u} , ( M_{\text{min}} )  $. Let $\tau : \mathfrak{g} \rightarrow \mathfrak{g} $ be the conjugation induced by the real form $ \mathfrak{g}_{0}  \subseteq \mathfrak{g}$. We call $ \tau ( \mathfrak{p} )$ \emph{the opposite parabolic subalgebra} from $ \mathfrak{p}$ and write \[ \frak{p}^{\text{op}}  = \tau ( \mathfrak{p} ) .\] If $\mathfrak{l}$ is the complexified Lie algebra of $L_{0}$ then \[ \mathfrak{p}  \cap  \tau ( \mathfrak{p} )  = \mathfrak{l}. \] In particular, \[ \mathfrak{u}^{\text{op}} = \tau ( \mathfrak{u}) \] is the nilradical of $ \frak{p}^{\text{op}} $.

\begin{proposition} 

Maintain the previous notations. Then there is a natural isomorphism of $L_{_{0}}$-modules: \[ H_{p} ( \mathfrak{u} , ( M_{\text{min}}   )^{h} )  \cong  H^{p} ( \mathfrak{u}^{\text{op}}, M_{\text{min}} )^{h}  \]

\end{proposition}

\begin{proof} 

We claim that \[ \mathfrak{u}^{h} \cong \mathfrak{u} \] as $L_{0}$-modules. Let $B (\xi, \zeta )$ denote the Killing form on $\frak{g}$. Then the Hermitian form on $\frak{u}$ defined by \[ \langle  \xi, \zeta \rangle = B (\xi, \tau (\zeta ) ) \] is nondegenerate and $L_{0}$-invariante, which proves the claim. Now an argument using  the standard complex for computing the $\frak{u}$-homology, as in the proof of Proposition 2.1 in \cite{bratten3}, shows that \[ H_{p} ( \mathfrak{u} , ( M_{\text{min}}   )^{h} )  \cong  H^{p} ( \mathfrak{u}^\ast, M_{\text{min}} )^{h}  \] where $ \mathfrak{u}^\ast $ is the dual representation. Hence the result follows, since \[ \mathfrak{u}^\ast \cong \mathfrak{u}^{\text{op}} \]. 

\end{proof}

\begin{corollary}

Let $I ( \mathfrak{p} , V )_{\text{min}} $ denote the representation $ H_{\text{c}}^{q} (S ,\mathcal{O} ( \mathfrak{p}, V_{\text{min}} ))$. Then  \[   \text{Hom}_{G_{0}} \left( I ( \mathfrak{p} , V)_{\text{min}}   , (I ( \mathfrak{p} ,  V )_{\text{min}} )^{h} \right)  \cong  \text{Hom}_{L_{0}} \left(  V_{\text{min}}  , H^{q} (\mathfrak{u}^{\text{op}} , I ( \mathfrak{p} , V)_{\text{min}} )^h \right) .  \]

\end{corollary}

\begin{example} Suppose $G_0$ is a connected complex reductive group and let $\mathfrak{b}$ be a nice Borel subalgebra. Let $C_0$ be the corresponding Cartan subgroup of $G_0$ and $\mathfrak{n}$ the nilradical of $\mathfrak{b}$. Suppose \[\chi_\mu : H_0 \rightarrow \mathbb{C}^\ast \] is a continuous character with derivative $\mu_0 \in (\mathfrak{c}_0)^\ast$ and complexified derivative $\mu \in \mathfrak{c}^\ast $. Let $\rho$ be one-half the sum of the roots of $\mathfrak{c}$ in $\mathfrak{b}$ and define $\lambda = \mu - \rho$. We assume $\lambda$ is antidominant and regular. In this case it is not hard to explicitly calculate the $C_0$-representation on the $\mathfrak{n}^{\text{op}}$-cohomology group  \[ H^{q} (\mathfrak{n}^{\text{op}} , I ( \mathfrak{b} , \mathbb{C}_\mu  )_{\text{min}} ) ) .\] In fact \[ H_q ( \mathfrak{n} , I ( \mathfrak{b} , \mathbb{C}_\mu  )_{\text{min}} ) \cong H^{q} (\mathfrak{n}^{\text{op}} , I ( \mathfrak{b} , \mathbb{C}_\mu  )_{\text{min}} ) ) .\] We briefly elaborate. First observe that since the maximal compact subgroup $K_0$ of $G_0$ is a compact real form, the complexification $K$ is a complex group isomorphic to $G$, although the actions of the two groups on the full flag space $X$ are distinct.  Thus $Q = K \cdot \mathfrak{b} $ is the full flag manifold of Borel subgroups of $G_0$ and this means the complex dimension  $s$ of $Q$ and the vanishing number $q$ coincide. Next, observe that the Weyl group, $W_0$ of $G_0$ acts simply transitively on the Borel subaglebras in $S$ that contain $\mathfrak{c}$ and that \[ w_l \cdot \mathfrak{b} = \mathfrak{b}^{\text{op}} \] if $w_l$ is the longest element in $W_0$. 

For each $w \in W_0$ let  \[\chi_{w \lambda + \rho} : H_0 \rightarrow \mathbb{C}^\ast \] be the character with complexified derivative $ w ( \mu - \rho ) + \rho .$ We claim  \[ H_q ( \mathfrak{n} , I ( \mathfrak{b} , \mathbb{C}_\mu  )_{\text{min}} ) \cong \bigoplus_{ w \in W_0 } \mathbb{C}_{ w \lambda + \rho } .  \] To establish this, let $W$ be the Weyl group of $\mathfrak{c}$ in $\mathfrak{g}$. Then, since the infinitesimal character of $I ( \mathfrak{b} , \mathbb{C}_\mu  ) $ is regular, one knows\[  H_q ( \mathfrak{n} , I ( \mathfrak{b} , \mathbb{C}_\mu  )_{\text{min}} ) = \bigoplus_{ w \in W } H_q ( \mathfrak{n} , I ( \mathfrak{b} , \mathbb{C}_\mu  )_{\text{min}} )_{w \lambda + \rho }  \] where the subindex indicates the corresponding eigenspace for the $\mathfrak{c}$-action. To calculate these eigenspaces we use the following two ingredients: (1) The intertwining functor (see \cite[Proposition 9.10]{hecht1}); (2) for each $w \in W_0$ there is a natural action on the standard complex for computing $\mathfrak{n}$-homology, which induces a linear isomorphism \[ H_q ( \mathfrak{n} , I ( \mathfrak{b} , \mathbb{C}_\mu  )_{\text{min}} ) \rightarrow H_q ( w^{ -1 }   \mathfrak{n} , I ( \mathfrak{b} , \mathbb{C}_\mu  )_{\text{min}} ) \] that intertwines the $C_0$ action on $  H_q ( \mathfrak{n} , I ( \mathfrak{b} , \mathbb{C}_\mu  )_{\text{min}} )  $ with the $ w^{ -1 }  C_0 $-action on \newline $ H_q ( w^{ -1 } \mathfrak{n} , I ( \mathfrak{b} , \mathbb{C}_\mu  )_{\text{min}} )     $.  

Calculating with the intertwining functor we can show \[  H_q ( \mathfrak{n} , I ( \mathfrak{b} , \mathbb{C}_\mu  )_{\text{min}} )_{w \lambda + \rho } = \lbrace 0 \rbrace \; \text{ if } \; w \notin W_0 \] and \[ H_q ( w^{-1} \mathfrak{n} , I ( \mathfrak{b} , \mathbb{C}_\mu  )_{\text{min}} )_{\lambda + w^{-1} \rho } \cong \mathbb{C}_{\mu - \rho + w^{-1} \rho } \; \text{ if } w \in W_0 . \] Hence \[  H_q ( \mathfrak{n} , I ( \mathfrak{b} , \mathbb{C}_\mu  )_{\text{min}} )_{w \lambda + \rho} \cong \mathbb{C}_{ w \lambda + \rho } \; \text{ if } \; w \in W_0 . \]Now the result follows since \[  H^q ( \mathfrak{n}^{\text{op}}, I ( \mathfrak{b} , \mathbb{C}_\mu  )_{\text{min}} ) \cong H_q ( \mathfrak{n}^{\text{op}}, I ( \mathfrak{b} , \mathbb{C}_\mu  )_{\text{min}} ) \otimes \mathbb{C}_{ 2 \rho}  = \]
\[ H_q ( w_l \mathfrak{n}, I ( \mathfrak{b} , \mathbb{C}_\mu  )_{\text{min}} ) \otimes \mathbb{C}_{ 2 \rho}  \cong H_q( \mathfrak{n} ,  I ( \mathfrak{b} , \mathbb{C}_\mu  )_{\text{min}} ). \]

It's worthwhile noting that the character \[ \chi_{\rho - w \rho} : C_0 \rightarrow S^1 \subseteq \mathbb{C}^\ast \] is unitary. In particular if $\alpha$ is a root of $\mathfrak{c}$ in $\mathfrak{b}$, appearing in the sum of roots given by $\rho - w \rho  $ then the root that takes the  value $ - \overline{\alpha}$ on $\mathfrak{c}_0 $ also appears in the sum. Thus, in the corresponding product of characters of $C_0$ we have $\chi_{\alpha}$ and $ \chi_{- \overline{\alpha}} = \chi_{\alpha}^h .$ Thus \[ \vert \chi_{\alpha} (g) \chi_{- \overline{\alpha}} (g) \vert = 1\] for each $g \in C_0$.

Now we can use the use the Frobenius reciprocity. In particular, if $\chi_\mu$ is a unitary character then there is a non degenerate continuous invariant Hermitian form on the irreducible representation $I ( \mathfrak{b} , \mathbb{C}_\mu  )_{\text{min}}  $ and each of the characters 
$\chi_{w \lambda + \rho}$ is unitary. On the other hand, if there is a nonzero continuous invariant Hermitian form on $ I (\mathfrak{b} , \mathbb{C}_\mu )_{\text{min}} $, then since $ I (\mathfrak{b} , \mathbb{C}_\mu )_{\text{min}} $ is irreducible, the corresponding nonzero morphism \[ I (\mathfrak{b} , \mathbb{C}_\mu )_{\text{min}}  \rightarrow  (I (\mathfrak{b} , \mathbb{C}_\mu )_{\text{min}} )^h \] induces an isomorphism on the level of Harish-Chandra modules. Thus \[ H_{q} (\mathfrak{n}^{\text{op}} , I ( \mathfrak{b} , \mathbb{C}_\mu  )_{\text{min}} )  \cong H_{q} (\mathfrak{n}^{\text{op}} , (I ( \mathfrak{b} , \mathbb{C}_\mu  )_{\text{min}} )^h ) \cong  H^{q} (\mathfrak{n}^{\text{op}} , I ( \mathfrak{b} , \mathbb{C}_\mu  )_{\text{min}}  )^h \] and we have \[ H_{q} (\mathfrak{n} , I ( \mathfrak{b} , \mathbb{C}_\mu  )_{\text{min}} )  \cong H_{q} (\mathfrak{n} , I ( \mathfrak{b} , \mathbb{C}_\mu  )_{\text{min}} ) ^h . \]

\end{example}

In the next section we will give a general consideration of the $L_0$-representation $H^{q} (\mathfrak{u}^{\text{op}} ,  H_{\text{c}}^{q} (S ,\mathcal{O} ( \mathfrak{p}, V_{\text{min}} )))^{h}$ and draw some conclusions.  

\section{Invariant Hermitian forms on representations}

In order to complete our consideration of the continuous invariant Hermitian forms on $H_{\text{c}}^{q} (S ,\mathcal{O} ( \mathfrak{p} , V_{\text{min}} ))  $, we will show in this section that \[ V_{\text{min}} \cong   H^{q} (\mathfrak{u}^{\text{op}} ,  H_{\text{c}}^{q} (S ,\mathcal{O} ( \mathfrak{p} , V_{\text{min}} )))_{\sigma}  \] at least when $V_{\text{min}}$ is irreducible and $G_0$ is a real form of a connected complex reductive group $G$. Our proof relies on the extension of a duality theorem, shown for discrete series in \cite{zabcic} and for open orbits in full flag manifolds in \cite{chang}. In particular, let $I (\mathfrak{p} , V) $ denote the Harish-Chandra module of $H_{\text{c}}^{q} (S ,\mathcal{O} ( \mathfrak{p}, V_{\text{min}} ))  $ and let $\mathfrak{p}^{\text{op}}$ be the parabolic subalgebra of $\mathfrak{g}$ opposite to $\mathfrak{p}$. The duality theorem that interests us, is simply expressed as an isomorphism of Harish-Chadra modules \[ I ( \mathfrak{p} ,V )^{\vee} \cong  I ( \mathfrak{p}^{\text{op}} ,V^{\vee}) . \] The proof utilized in \cite{chang} carries over in a rather straightforward manner to the context of a very nice parabolic subgroup and an irreducible representation. However, this  demonstration doesn't produce a natural transformation of (exact, contravariant) functors. The proof in \cite{zabcic} does define a natural transformation of functors, but it is not clear to us how to generalize this.   

We need to adapt our notation to the design of the argument in \cite{chang}  and return to the conventions in Section 2. Let $\mathfrak{h}^\ast$ denote the abstract Cartan dual and for $\lambda \in \mathfrak{h}^\ast$, let $\sigma_\lambda$ denote the corresponding infinitesimal in $\mathcal{Z} ( \mathfrak{l} )$. Since the infinitesimal character $\sigma$ for the $L_0$-module $V_{\text{min}} $ is antidominant and regular with respect to the parabolic subalgebra $\mathfrak{p}$, there is a $\lambda \in \mathfrak{h}^\ast$ which is antidominant and regular, such that $\sigma$ is the specialization of $\sigma_\lambda $ to $\mathfrak{l}$ in $\mathfrak{p}$. Let $w_l$ be the longest element in the Weyl group of $\mathfrak{h}^\ast$. Then the abstract infinitesimal character $\sigma_{ - w_l \lambda}$ is antidominant and regular. If $\sigma^{\vee} $ denotes the infinitesimal character of $V^{\vee}$ then the specialization of  $\sigma_{ - w_l \lambda}$ to $\mathfrak{l}$ at $\mathfrak{p}^{\text{op}}$ is $\sigma^{\vee} $. In general, if $M$ is a $\mathfrak{g}$-module, we let \[ H_p ( \mathfrak{u}_y , M )_{\sigma_{ - w_l \lambda} } \] denote the corresponding $Z(\mathfrak{l}_y)$-eigenspace on the $p$-th $\mathfrak{u}_y$-homology group given by specialization. Let $I$ denote the Harish-Chandra module $I (\mathfrak{p} ,V) $. Since $\lambda$ is a parameter for the infinitesimal character of $I$ it follows that $  - w_l \lambda$ is a parameter for the infinitesimal character of $I^\vee$. The idea of the proof is to calculate the localization \[  \Delta_{- w_l \lambda } ( I^\vee ) \] to $Y$ with respect to $\sigma_{ - w_l \lambda} $ by calculating the geometric fibers, which are given by the homology groups 
\[ H_p ( \frak{u}_y , I^{\vee} )_{ \sigma_{- w_l \lambda}} .\] Similarly, we can directly calculate the corresponding analytic localization of $ (I^{\vee})_{\text{min}}$ to obtain the result. Actually, the calculation on the homology group $H_q(\frak{u}^{\text{op}}, I^{\vee})_{\sigma^\vee} $ is sufficient for our purposes, but the argument for the full duality theorem is a simple step more. 

To calculate the homology groups we use a construction which Chang credits to Vogan \cite[Lemma 7.2]{chang}. In particular, suppose $G_0$ is a connected real form of a  connected complex reductive Lie group $G$ and $K_0$ is a maximal compact subgroup of $G_0$. Let $C \subseteq G$ be a Cartan subgroup such of $G$ that $G_0 \cap C = C_0$ is a real form of $C$ and such that $K_0 \cap C_0 = T_0$ is a maximal torus of $C_0$ ($C$ is called \emph{stable}). Then there exists an involutive automorphism \[  \theta_C : G \rightarrow G \] such that $\theta_C (K_0) = K_0$ and such that $\theta_C (h) = h^{-1} $ for every $h \in C$. 

Given a Harish-Chandra module $M$ for $ (\mathfrak{g} , K_0 )$ we can define a new Harish-Chandra module $M^{\theta_C}$ by applying first the automorphism and then the $ (\mathfrak{g} , K_0 )$-action on $V$. The property we will use is when $M$ is irreducible, then \[ M^{\theta_C} \cong M^\vee .\]

To calculate the homology groups $H_p(\frak{u}^{\text{op}}, I^{\vee})_{\sigma^\vee} $ we let \[ L = P \cap P^{\text{op}} \] and choose a stable Cartan subgroup $C \subseteq L$. Thus $\theta_C (L) = L $ and $\theta_C (K_0 \cap L_0) = K_0 \cap L_0$. Also note that the map $\theta_C: \mathfrak{u} \rightarrow  \mathfrak{u}^{\text{op}}$ defines an isomorphism of $L$-modules \[ \frak{u}^{\theta_C} \cong \mathfrak{u}^{\text{op}} .\] Applying the automorphism $\theta_C$ to the standard complex for computing homology, and using the fact that $ \sigma \circ \theta_C = \sigma^\vee $ we obtain \[ H_p(\frak{u}^{\text{op}}, I^{\vee})_{\sigma^\vee} \cong  H_p(\frak{u}^{\theta_C}, I^{\theta_C})_{\sigma^\vee} \cong (H_p(\frak{u} , I)_{\sigma})^{\theta_C}  \cong \begin{cases} \lbrace 0 \rbrace &  p \neq q \\ V^{\theta_C} \cong V^\vee & p \neq q   \end{cases} \] since $V$ is irreducible. 

On the hand, suppose $\mathfrak{p}_y \in Y$ and $C  \subseteq P_y$ is a stable Cartan subgroup such that $\theta_C (\mathfrak{p}_y) = \mathfrak{p}_y^{ \text{op}}  \notin S$. We choose a Levi factor $L$ of $P_y$, that contains $C$ and such that $K_0 \cap P_y \subseteq L$ and let $\mathfrak{l}$ be the Lie algebra of $L$. Then $H_p ( \mathfrak{u}_y , I^\vee )$ is a Harish-Chandra module (it has finite length and is admissible) for $ (\mathfrak{l} , K_0 \cap L)$. As before, there is an isomorphism of $L$-modules \[ \mathfrak{u}_y^{\theta_C}  \cong \mathfrak{u}_y^{ \text{op}} \]  and the infinitesimal character for $Z(\mathfrak{l}) $ given by the composition of the specialization of $\sigma_{ - w_l \lambda}$ to $y$ with $\theta_C$ is equal to the specialization of $\sigma_\lambda$ to the opposite point. Thus we have  \[  H_p ( \mathfrak{u}_y , I^\vee )_{\sigma_{ - w_l \lambda}} \cong H_p ( (\mathfrak{u}_y^{ \text{op}})^{\theta_C} , I^{ \theta_C} )_{\sigma_{ - w_l \lambda}} \cong ( H_p ( \mathfrak{u}_y^{ \text{op}}  , I )_{\sigma_\lambda} )^{ \theta_C} = 0 . \] 

\begin{theorem} 

Suppose $G_0$ is a real form of a connected complex reductive group $G$. Let $P$ be very nice parabolic subgroup and suppose $V_{\text{min}} $ is an irreducible minimal globalization for $L_0$ with an antidominant and regular infinitesimal character. Let $I (P ,V) $ be the Harish-Chandra module of the irreducible minimal globalization $H_{\text{c}}^q( S , \mathcal{O} ( P , V_{\text{min}} )) $ and let $I (P^{\text{op}} ,V^\vee) $ be the Harish-Chandra module of minimal globalization $H_{\text{c}}^q( S , \mathcal{O} ( P^{\text{op}} , (V^\vee)_{\text{min}})) .$ Then \[ I (P ,V)^\vee  \cong  I (P^{\text{op}} ,V^\vee) .\]

\end{theorem}  

\begin{proof} When $G_0$ is connected, then the previous discussion establishes the formula in the following way. Letting $I = I(P,V)$, then the calculations on the Lie algebra homology groups determine the Harish-Chandra modules of the geometric fibers of the analytic localization  $L\Delta_{Y , \sigma_{ - w_l \lambda} }  ( (I^\vee)_{\text{min}} )$ of $(I^\vee)_{\text{min}}$ to $Y$ with respect to the abstract infinitesimal character $\sigma_{ - w_l \lambda}$. Hence \[  L\Delta_{Y , \sigma_{ - w_l \lambda} }  ( (I^\vee)_{\text{min}} ) \cong \mathcal{O} (P^{\text{op}} , (V^\vee)_{\text{min}}) \mid_{S^{\text{op}}} [ q ] \] where $\mathcal{O} ( P , (V^\vee)_{\text{min}}) \mid_{S^{\text{op}}}$ is the extension by zero of  the analytic sheaf $\mathcal{O} ( P , (V^\vee)_{\text{min}})$ on $S^{\text{op}}$ to all of $Y$. Thus the hypercohomology of the of the analytic localization provides the isomorphism \[ (I^\vee)_{\text{min}} \cong H_{\text{c}}^q( S , \mathcal{O} ( P^{\text{op}} , (V^\vee)_{\text{min}})) .\] 

For the general case, let $(G_0)^e$, $(L_0)^e$ be the respective identity components for $G_0$ and $L_0$.  Then $(G_0)^e \cap P = (L_0)^e$ so that $(K_0)^e \cap P = (K_0 \cap L_0)^e$. In particular, we have a natural inclusion \[ K_0 \cap L_0 / (K_0 \cap L_0)^e \rightarrow K_0 / (K_0)^e .  \] 

The Harish-Chandra $V$ is a finite sum of irreducible $\mathfrak{l}$-modules $V_m$ so we can choose a (finite) set of representatives $\lbrace k_j \rbrace$ for the coclases of $(K_0 \cap L_0)^e$ in $K_0 \cap L_0$ and \[V = \oplus V_m \] where the $K_0 \cap L_0$-action  is characterized by the $(K_0 \cap L_0)^e$-action on the irreducible $\mathfrak{l}$-modules $V_m$ and by the way the operators corresponding to the elements $k_j$ permute the irreducible modules $V_m$. Thus \[  \Gamma ( P , V)^\vee = \Gamma ( P , \oplus V_m)^\vee \cong \oplus \; \Gamma (P , V_m)^\vee   \] where the $K_0$-action on $ \oplus \; \Gamma (P , V_m)^\vee  $ is characterized by the $(K_0)^e$-action on the irreducible $\mathfrak{l}$-modules $\Gamma (P , V_m)^\vee$ and by the way the operators $k_j$ permute these $\mathfrak{l}$-modules. 

Similarly \[ \Gamma ( P^{\text{op}} , V^\vee ) = \Gamma ( P^{\text{op}} , \oplus (V_m)^\vee ) \cong  \oplus \; \Gamma ( (P^{\text{op}} , (V_m)^\vee ) .\] Hence the theorem is true, since \[ \Gamma ( (P^{\text{op}} , (V_m)^\vee ) \cong \Gamma (P , V_m)^\vee . \]

\end{proof} 

\begin{theorem} 

Suppose $G_0$ is a real form of a connected complex reductive group $G$. Let $P$ be a nice parabolic subgroup and suppose $V_{\text{min}} $ is an irreducible minimal globalization for $L_0$ with an antidominant and regular infinitesimal character. Then a nondegenerate continuous invariant Hermitian form on $V_{\text{min}} $ naturally induces a a nondegenerate continuous invariant Hermitian form on $ H_q ( S , \mathcal{O} (P , V_{min}))$.  

\end{theorem}

\begin{proof} 

We can choose a maximal compact subgroup $K_0$ of $G_0$ such that $P$ is a very nice parabolic subgroup. Let $I = I (P ,V)$ be the Harish-Chandra module of $ H_q ( S , \mathcal{O} (P , V_{min}))$. From our previous work we have \[ V^\vee \cong H_q ( \mathfrak{u}^{\text{op}} , I ( P^{\text{op}} , V^\vee ) )_{\sigma^\vee} \cong H_q ( \mathfrak{u}^{\text{op}} , I^\vee )_{\sigma^\vee} .\]

Using the result in \cite{bratten3}, we have \[  H_q ( \mathfrak{u}^{\text{op}} , I^\vee ) \cong H^q ( \mathfrak{u}^{\text{op}} , I )^\vee  \Rightarrow V^\vee \cong (H^q (  \mathfrak{u}^{\text{op}} , I )_\sigma)^\vee . \] Hence \[ V \cong H^q ( \mathfrak{u}^{\text{op}} , I )_\sigma \] and therefore \[ (V_{\text{min}})^h \cong (H^q ( \mathfrak{u}^{\text{op}} , I_{\text{min}} )_\sigma)^h  \cong H_q ( \mathfrak{u} , ( I_{\text{min}} )^h )_{\sigma^h } . \] It follows that a nondegenerate Hermitian form (in fact a nonzero Hermitian form) induces a nonzero morphism \[  V_{\text{min}} \rightarrow H_q ( \mathfrak{u} , ( I_{\text{min}} )^h ) .\] By the Frobenius reciprocity we obtain a nonzero morphism \[ I_{\text{min}}  \rightarrow (I_{\text{min}} )^h   \] which is injective, since $I_{\text{min}}$ is irreducible. The result follows since the space of continuous invariant Hermitian forms on $I_{\text{min}}$ is a real form of the space of continuous invariant sesqulinear forms.

\end{proof}

\end{document}